\documentclass{amsart}
\usepackage{color}
\usepackage[utf8]{inputenc}
\usepackage{amscd}
\usepackage{pstricks}
\usepackage{amsmath}
\usepackage[english]{babel}
\usepackage{amsfonts}
\usepackage{graphicx}
\usepackage{vmargin}
\usepackage{theoremref}
\usepackage[hidelinks]{hyperref}
\usepackage{amssymb}
\usepackage{tikz}
\usepackage{relsize}
\usepackage[toc,page]{appendix}
\usepackage{commath}
\usepackage{mathrsfs}
\usetikzlibrary{matrix}
\usepackage{cancel}
\usepackage{lineno}

\newtheorem{theorem}{Theorem}[section]
\newtheorem{lemma}[theorem]{Lemma}

\newtheorem{definition}[theorem]{Definition}

\newtheorem{corollary}[theorem]{Corollary}

\newtheorem{example}[theorem]{Example}

\newtheorem{remark}[theorem]{Remark}

\usepackage{chngcntr}
\counterwithin{figure}{section}

\numberwithin{equation}{section}

\title[]{A note on the differentiability of Palmer's topological equivalence for discrete systems}
\author[]{\'A. Casta\~neda, N. Jara*}

\address{Universidad de Chile, Departamento de Matem\'aticas. Casilla 653, Santiago, Chile}

\email{nestor.jara@ug.uchile.cl, castaneda@uchile.cl}
\subjclass[2010]{34D09, 37C60, 37D25}
\keywords{Nonautonomus hyperbolicity, Nonautonomous difference equation, Smooth linearization}
\thanks{This research has been partially supported by ANID, Beca Nacional de Magister 22200774.}
\thanks{This research has been partially supported by FONDECYT Regular 1200653}
\date{\today}

\begin{document}
\maketitle

\begin{abstract}
    A linear system of difference equations and a nonlinear perturbation are considered. 
    We propose sufficient conditions to ensure that the homeomorphism of topological equivalence between them is actually a $C^1$ diffeomorphism. These conditions consider that the linear part satisfies a dichotomy on the positive half-line, while the perturbation satisfies boundness and Lipschitzness conditions, and anothers that are tailored for our goal. The family of dichotomies that we study consider (but are not limited to) the exponential and nonuniform exponential dichotomies. We discuss the possibility of extending this result to a higher class of differentiability.
\end{abstract}

\section{Introduction}

The concept of topological equivalence between two systems, linear and nonlinear, has played an important role in the study of differential equations, mainly due to the difficulty to find analytical solutions for them. It has been more fructiferous to study a linear, simpler system and from it obtain information about the nonlinear original system. Some facts in this analysis are the search for bounded solutions and characterization of their asymptotic behaviour. The method consists in establishing a continuous correspondence between solutions of both systems.

The problem of linearization of the flow of solutions of ordinary autonomous differential equations began in the decade of 1960 with the classical P. Hartman's theorem \cite{Hartman}, which ensures the existence of a local homeomorphism between a nonlinear system and and its linearization around a fixed point, under the assumption that the linear system satisfies an hyperbolicity condition. During the next decade the result was generalized to the construction of a global homeomorphism thanks to C. Pugh \cite{Pugh} and A. Reinfeld \cite{Reinfeld}.\\

The extension of the previous result is due to result of K. J. Palmer \cite{Palmer}, whom in 1973 studied a couple of continuous systems given by: 

\begin{equation}\label{89}
    \dot{x}=A(t)x
\end{equation}

\begin{equation}\label{90}
    \dot{y}=A(t)y+f(t,y).
\end{equation}

He used the exponential dichotomy property associated to the nonautonomous linear system, which mimics the hyperbolocity of the autonomous framework, namely, establish the existence of stable and unstable manifolds for solutions. Moreover, Palmer's result considers a family of nonlinear perturbations for the system (\ref{90}) that satisfy smallness and Lipschitz properties. Palmer's result was the first to establish the topological equivalence (a more precise definition will be provided later) between the systems (\ref{89}) and (\ref{90}). However, most often a exponential dichotomy represents highly restrictive hypothesis, which is why during the next decade  \cite{Palmer2}, \cite{Palmer3}  the same author generalized his result to consider a wider family of dichotomies and perturbations.

Contemporaneously to K. J. Palmer, it was already being developed the analogous theory in the discrete framework (we refer the reader to C. Coffman \cite{Coffman}). These contributions were retaken in lately 80's and early 90's, when G. Papaschinopoulos's work (\cite{Papas1}, \cite{Papas2}, \cite{Papas3}, \cite{Papas4}) was very prominent and helped to characterize dichotomies and asymptotic behavior of solutions.

In the decade of 1990 many authors came back to study K. J. Palmer's problem in its continuous version, admitting wider hypothesis. Among them we recall the contributions of  Z. Lin and Y-X. Lin \cite{Lin}, L. Jiang \cite{Jiang} and J.L. Shi and K.Q. Xiong \cite{Shi}.

In relation to new dichotomies, M. Pinto and R. Naulin (\cite{Pinto1}, \cite{Pinto2}) stand out for introducing the concept of $(h,k)-$dichotomies. We also highlight L. Barreira and C. Valls (\cite{Barreira1}, \cite{Barreira2}, \cite{Barreira3}), whom in early 2000's extended the family of dichotomies, introducing the concept of nonuniform dichotomy. During the last decade these concepts have been used by many authors in both the discrete \cite{Crai}, and continuous framework \cite{Bento}, even in integrated contexts as impulsive equations \cite{Fenner}.

Since 2015, authors \'A. Casta\~{n}eda and G. Robledo \cite{Castaneda4} have studied conclusions that can be obtained when studying the differentiability of this homeomorphism, particulary when restricting the problem to the half real line. 

Other authors such as D. Dragi$\check{c}$evi\'c \textit{et al.}  \cite{Dragicevic2, Dragicevic} have also been interested in these properties of differentiability in a continuous and discrete context respectively . Recently, faced with the problem in its discrete context, the authors on \cite {Castaneda3} have studied the differentiability of said homeomorphism under the assumption that the linear system satisfies a nonuniform contraction.

\textbf{Novelty of this work}. Our work follows the results initiated by \'A. Casta\~{n}eda and G. Robledo in \cite{Castaneda4}, and is heavily based on their continuation along with P. Gonz\'alez in \cite{Castaneda3}, since we use some of their results and strategies to find a K.J. Palmer's homeomorphism between solutions of (\ref{98}) and (\ref{99}), and our results of topological equivalence are over $J=\mathbb{Z}^+$, instead of $\mathbb{Z}$. In the section of preliminaries we give precise definitions for these concepts and make explicit which tools and results we are considering from \cite{Castaneda3}. We also present the set of conditions under which we will work and make a brief discussion about their appearances in literature. The main difference between our work and \cite{Castaneda3}, in terms of differentiability, is that we allow the existence of nonempty unstable manifolds for the linear system; furthermore, we do not use spectral theory in order to obtain $C^1-$differentiability such as it is employ in \cite{Dragicevic}. 

In the third section, we propose conditions that allows us to prove that K.J. Palmer's homeomorphism, which we previously studied, is a $C^1-$diffeomorphism. We give examples in order to achieve these conditions, including systems that satisfy exponential and nonuniform exponential dichotomies.

In the fourth section we study a generalization of the previous conditions in order to obtain higher order derivatives. Finally, in the last section, we propose an algebraic set of functions and an operator that can be inductively applied in order to find conditions for arbitrary high order derivatives.

\section{Preliminaries}

In this work, we study the following nonautonomous discrete systems:
\begin{equation}\label{98}
    x(k+1)=A(k)x(k)
\end{equation}
\begin{equation}\label{99}
    y(k+1)=A(k)y(k)+f(k,y(k)),
\end{equation}
where $A:\mathbb{Z}^+\to \mathcal{M}_{d\times d}(\mathbb{R})$ and $f:\mathbb{Z}^+\times \mathbb{R}^d\to \mathbb{R}^d$, satisfying properties that will be given later. Their solutions in each case is a function $x,y:\mathbb{Z}^+\to \mathbb{R}^d$, and we denote $n\mapsto x(n,k,\xi)$ and $n\to y(n,k,\eta)$ for the solutions of (\ref{98}) and (\ref{99}) that pass through $\xi$ and $\eta$ respectively on $n=k$.\\

Furthermore, we will consider the following properties:\\
\begin{itemize}
   
    \item [\textbf{(d0)}] $A:\mathbb{Z}^+\to M_{d\times d}(\mathbb{R})$ is non singular (\textit{i.e.} has invertible images) and uniformly bounded, that is to say, there exists $M> 1$ such that:
    \begin{equation*}
    \max\left\{ \sup_{k \in \mathbb{Z}^+}\norm{A(k)}, \sup_{k\in \mathbb{Z}^+}\norm{A^{-1}(k)}\right\}=M.
    \end{equation*}

    \item [\textbf{(d1)}] The system (\ref{98}) has a non uniform dichotomy, \textit{i.e.} there are two invariant complementary projectors $P(\cdot)$ and $Q(\cdot)$ such that $P(n)+Q(n)=I$ for every $n\in \mathbb{Z}^+$, a non-negative sequence $D$ and a monotone decreasing sequence convergent to zero $h$ with $h(0)=1$ such that:
    $$\left\{ \begin{array}{lc}
            \norm{\Phi(k,n)P(n)}\leq D(n)
            \left( \mathlarger{\frac{h(k)}{h(n)}}\right), &\forall k \geq n\geq 0 \\
            \\ \norm{\Phi(k,n)Q(n)}\leq D(n)
            \left( \mathlarger{\frac{h(n)}{h(k)}}\right), &\forall 0\leq k\leq n,
             \end{array}
   \right.$$
   where $\Phi(k,n)$ is the transition matrix for (\ref{98}).
   
   \item [\textbf{(d2)}] There exist $\mu, \gamma:\mathbb{Z}^+\to \mathbb{R}^+$ sequences such that for every $k\in \mathbb{Z}^+$ and every pair $(y,\Tilde{y})\in \mathbb{R}^d\times \mathbb{R}^d$ we have:
    $$|f(k,y)-f(k,\Tilde{y})|\leq \gamma(k)|y-\Tilde{y}|\text{ ; }|f(k,y)|\leq \mu(k).$$
    
    \item [\textbf{(d3)}] The sequence $\mu$ defined above verifies:
    $$\sum_{j=0}^{\infty}\norm{\mathcal{G}(k,j+1)}\mu(j)\leq p<\infty\text{ for every }k\in \mathbb{Z}^+,$$
    where $\mathcal{G}$ is the Green's operator asociated to the dichotomy established in \textbf{(d1)}.
    
    \item [\textbf{(d4)}] The sequence $\gamma$ defined above verifies:
    $$\sum_{j=0}^{\infty}\norm{\mathcal{G}(k,j+1)}\gamma(j)\leq q<1\text{ for every }k \in \mathbb{Z}^+.$$

    \item [\textbf{(d5)}] The sequence $\gamma$ and the matrix function $A$ verify:
    $$\norm{A^{-1}(\ell)\gamma(\ell)}<1\text{ for every }\ell \in \mathbb{Z}^+.$$
    
    \item [\textbf{(d6)}] The map $u\mapsto f(k,u)$ and its derivates respect to $u$ up to the order $r$ ($r\geq 1$) are continuous functions of $(k,u)\in \mathbb{Z}^+\times \mathbb{R}^d$ and $\sup_{u\in \mathbb{R}^d}\norm{\frac{\partial f}{\partial u}(k,u)}<+\infty $ is bounded. 
    
    \item [\textbf{(d7)}] For each fixed $m\in \mathbb{Z}^+$, the sequences $D$, $h$ and $\gamma$ satisfy:
    $$\mathlarger{\sum_{j=m}^{\infty}}\left(D(j+1)h(j+1)\gamma(j)\left[ \mathlarger{\prod_{p=m}^{j-1}}\norm{A(p)}+\gamma(p)\right]\right)<+\infty.$$
\end{itemize}

\begin{remark}
The transition matrix for (\ref{98}) is the matrix function $\Phi:\mathbb{Z}^+\times\mathbb{Z}^+\to \mathcal{M}_{d\times d}(\mathbb{R})$ given by:
\begin{equation*}
    \Phi(k,n)= \left\{ \begin{array}{lcc}
             A(k-1)A(k-2)\cdots A(n) &  \text{if } & k > n \\
             \\ I & \text{ if }& k=n \\
             \\ A^{-1}(k)A^{-1}(k+1)\cdots A^{-1}(n-1) & \text{if } & k<n.
             \end{array}
   \right.
\end{equation*}
\end{remark}

\begin{remark}
Projectors $P(\cdot)$ and $Q(\cdot)$ have been called invariant, which means that they are under the assumption that for every  $k,n\in \mathbb{Z}^+$ they satisfy:
$$P(k)\Phi(k,n)=\Phi(k,n)P(n)\text{ and }Q(k)\Phi(k,n)=\Phi(k,n)Q(n).$$
\end{remark}

\begin{remark}
Green's operator for the system (\ref{98}) asociated to the dichotomy \textbf{(d1)} is the matrix function $\mathcal{G}:\mathbb{Z}^+\times\mathbb{Z}^+\to \mathcal{M}_{d\times d}(\mathbb{R})$ given by:
\begin{equation*}
    \mathcal{G}(k,n)= \left\{ \begin{array}{cc}
             \Phi(k,n)P(n)  & \forall k \geq n\geq 0 \\
             \\ -\Phi(k,n)Q(n) & \forall 0\leq k<n.
             \end{array}
   \right.
\end{equation*}

And it is easily deduced that it satisfies:
$$\mathcal{G}(k+1,n)=A(k)\mathcal{G}(k,n)\text{ and }\mathcal{G}(n,n)=I+A(n-1)\mathcal{G}(n-1,n).$$
\end{remark}

\begin{remark}\label{185}
Condition \textbf{(d5)} was used in \cite{Castaneda3} in order to ensure that solutions for the system (\ref{99}) can be uniquely backwards continued.

However, we have noted that some authors (\cite{Dragicevic}, for example) assume that this backwards continuation can be uniquely found even without this condition.
\end{remark}

Now, we recall the concept of topological equivalence for discrete systems.

\begin{definition}\label{102}

Let $J\subset \mathbb{Z}$ be a interval of integer numbers. The systems (\ref{98}) and (\ref{99}) are $J$-topologically equivalent if there is a function $H:J\times \mathbb{R}^d\to \mathbb{R}^d$ that satisfies:

\begin{itemize}
    \item [i)] If $x(k)$ is a solution of (\ref{98}), then $H[k,x(k)]$ is a solution of (\ref{99}).
    \item [ii)] $H(k,u)-u$ is bounded over $J\times \mathbb{R}^d$.
    \item [iii)] For each fixed $k\in J$, the map $u\mapsto H(k,u)$ is an homeomorphism of $\mathbb{R}^d$.
\end{itemize}

Moreover, the function $u\mapsto G(k,u)=H^{-1}(k,u)$ satisfies ii) and iii) and maps solutions of (\ref{99}) into solutions of (\ref{98}).
\end{definition}

However, as it was explained before, it turns out to be interesting to study which properties can be deduced when studying the differentiability of such homeomorphism. In the discrete case the characterization is even better than in the continuous framework, because as a difference equation may be understood as a dynamical system where a discrete group acts on the space $\mathbb{R}^d$, the space can be split on fibers, and so as soon as we find a diffeomorphism associated to one of such fibers, the other can be replicated inductively. This idea motivates the following definition:

\begin{definition}
The systems (\ref{98}) and (\ref{99}) are $C^r$-topologically equivalent on $\mathbb{Z}^+$ if they are topologically equivalent on $\mathbb{Z}^+$ and $u\mapsto H(k,u)$ is a diffeomorphism of class $C^r$, with $r\geq 1$ for each fixed $k\geq0$.
\end{definition}

Most of the conditions we have stated have been used several times in the literature. We now highlight some of the results that inspire this work.

\begin{theorem}\label{116} \cite[Theorem 2.1]{Castaneda3} If the conditions \textbf{(d0)-(d5)} hold, then the systems (\ref{98}) and (\ref{99}) are $\mathbb{Z}^+-$topologically equivalent.

\end{theorem}

In the proof of  \ref{116}, some calculation were developed and we highlight some of them in order to facilitate read of this article.
Let us define
\begin{eqnarray}\label{112}
w^*(k;(m,\eta))&=&-\sum_{j=0}^\infty \mathcal{G}(k,j+1)f(j,y(j,m,\eta)) \nonumber \\ 
\nonumber\\
&=&-\sum_{j=0}^{k-1} \Phi(k,j+1)P(j+1)f(j,y(j,m,\eta))\\
\nonumber\\
&&+\sum_{j=k}^{\infty} \Phi(k,j+1)Q(j+1)f(j,y(j,m,\eta)) \nonumber
\end{eqnarray}

On the other hand, for each $(m,\xi)\in \mathbb{Z}^+\times \mathbb{R}^d$ we define the map $\Theta:\ell^{\infty}(\mathbb{Z}^+,\mathbb{R}^d)\to \ell^{\infty}(\mathbb{Z}^+,\mathbb{R}^d)$ given by:
$$(\Theta \phi)(k;(m,\xi))=\sum_{j=0}^\infty \mathcal{G}(k,j+1)f(j,x(j,m,\xi)+\phi(j;(m,\xi)))$$

Using Banach's fixed point Theorem we conclude the existence of a unique fixed point
\begin{equation}\label{117}
    z^*(k;(m,\xi)))=\sum_{j=0}^{+\infty}\mathcal{G}(k,j+1)f(j,x(j,m,\xi)+z^*(j;(m,\xi))).
\end{equation}

It is easily verified that this map is a solution for the initial values problem:
\begin{equation}\label{106}
    \left\{ \begin{array}{ccl}
            z(k+1) & = & A(k)z(k)+f(k,x(k,m,\eta)+z(k))\\
            \\z(0) & = & - \mathlarger{\sum}_{j=0}^\infty \Phi(0,j+1)Q(j+1)f(j,x(j,m,\eta)+z^*(j;(m,\xi))).
\end{array}
\right.\end{equation}

By unicity of solutions we have
\begin{equation}\label{103}
    x(k,m,\xi)=x(k,p,x(p,m,\xi))\text{  , for every }k,p,m\in \mathbb{Z}^+.
\end{equation}

Analogously, we can verify
\begin{equation}\label{104}
    z^*(k;(m,\xi))=z^*(k;(p,x(p,m,\xi)))\text{  , for every }k,p,m\in \mathbb{Z}^+.
\end{equation}

For every fixed $k\in \mathbb{Z}^+$ we construct the maps $H(k,\cdot):\mathbb{R}^d\to \mathbb{R}^d$ and $G(k,\cdot):\mathbb{R}^d\to \mathbb{R}^d$ given by
\begin{equation}\label{105}\left\{ \begin{array}{ccl}
            H(k,\xi) & = & \xi + \mathlarger{\sum}_{j=0}^{+\infty} \mathcal{G}(k,j+1)f(j,x(j,k,\xi)+z^*(j;(k,\xi))) \\
            \\ & = & \xi+z^*(k;(k,\xi))
\end{array}
\right.\end{equation}

and
\begin{equation}\label{115}\left\{ \begin{array}{ccl}
            G(k,\eta) & = & \eta - \mathlarger{\sum}_{j=0}^{+\infty} \mathcal{G}(k,j+1)f(j,y(j,k,\eta)) \\
            \\ & = & \eta+w^*(k;(k,\eta)).
\end{array}
\right.\end{equation}

It has been proved that $H$ and $G=H^{-1}$ define the homeomorphism that shows the topological equivalence between (\ref{98}) and (\ref{99}). \\

In order to study additional properties of $G$, note that for $n<k$ we have:

\begin{equation}\label{118}
y(n,k,\eta)=\Phi(n,k)\eta-\sum_{j=n}^{k-1}\Phi(n,j+1)f(j,y(j,k,\eta)).
\end{equation}

Which implies that
\begin{eqnarray*}
\Phi(k,n)y(n,k,\eta)&=&\eta-\sum_{j=n}^{k-1}\Phi(k,j+1)f(j,y(j,k,\eta))\\
&=&\eta-\sum_{j=n}^{k-1}\Phi(k,j+1) P(j+1)f(j,y(j,k,\eta))\\
& &-\sum_{j=n}^{k-1}\Phi(k,j+1) Q(j+1)f(j,y(j,k,\eta)).
\end{eqnarray*}

In particular, for $n=0$ we have
\begin{eqnarray*}
\Phi(k,0)y (0,k,\eta)&=&\eta-\sum_{j=0}^{k-1}\Phi(k,j+1) P(j+1)f(j,y(j,k,\eta))\\
& &-\sum_{j=0}^{k-1}\Phi(k,j+1) Q(j+1)f(j,y(j,k,\eta))\\
&=&\eta-\sum_{j=0}^{k-1}\Phi(k,j+1) P(j+1)f(j,y(j,k,\eta))\\
& &+\sum_{j=k}^{\infty}\Phi(k,j+1) Q(j+1)f(j,y(j,k,\eta))\\
& &-\sum_{j=0}^{\infty}\Phi(k,j+1) Q(j+1)f(j,y(j,k,\eta))\\
&=&\eta-\sum_{j=0}^{\infty}\mathcal{G}(k,j+1)f(j,y(j,k,\eta))-\sum_{j=0}^{\infty}\Phi(k,j+1) Q(j+1)f(j,y(j,k,\eta))\\
&=&G(k,\eta)-\Phi(k,0)\sum_{j=0}^{\infty}\Phi(0,j+1) Q(j+1)f(j,y(j,k,\eta)).
\end{eqnarray*}

Hence, using the definition of $k\mapsto w^*(0;(k,\eta))$ we conclude
\begin{eqnarray}\label{107}
G(k,\eta)&=&\Phi(k,0)\left\{ y(0,k,\eta)+\sum_{j=0}^\infty \Phi(0,j+1)Q(j+1)f(j,y(j,k,\eta))\right\} \nonumber \\
\\ 
&=&\Phi(k,0)\left\{ y(0,k,\eta)+w^*(0;(k,\eta))\right\}. \nonumber
\end{eqnarray}

Nevertheless, in order to prove the $C^r-$topologically equivalence, those authors had to impose stronger conditions than before, having to remain to the stable manifold, \textit{i.e.}, supposing the system (\ref{98}) admits a nonuniform contraction. Thus, condition \textbf{(d1)} is called as:

\begin{itemize}
      \item [\textbf{(D1)}] The system (\ref{98}) has a nonuniform dichotomy with trivial projector, \textit{i.e.} there exist a nonnegative sequence $D$ and a monotone decreasing sequence convergent to zero $h$, with $h(0)=1$, such that
    \begin{equation*}
    \norm{\Phi(k,n)}\leq D(n)\frac{h(k)}{h(n)}\text{   ,    }\forall k\geq n\geq 0.
    \end{equation*}
\end{itemize}

Furthermore,  \textbf{(d6)} is also considered, hence obtaining the next result:

\begin{theorem}\label{137} 
\cite[Lemma 2.3]{Castaneda3} If the conditions \textbf{(d0), (D1)} and \textbf{(d2)-(d6)} hold, then the systems (\ref{98}) and (\ref{99}) are $C^r-$topologically equivalent on $\mathbb{Z}^+$.
\end{theorem}

\begin{remark}
Note that when replacing \textbf{(d1)} with \textbf{(D1)}, the form of Green's operator is simpler,and hence conditions \textbf{(d3)} and \textbf{(d4)} are also considerably easy to handle. Also, in this case, the map $k\mapsto w^*(0;(k,\eta))$ nullifies, which allowed the authors to prove more easily the differentiability of $u\mapsto G(k,u)$.
\end{remark}

\begin{remark}\label{191}
Note that Theorem \ref{137} and the expression (\ref{107}), implies that under conditions \textbf{(d0)}, \textbf{(d1)-(d6)}, the map $\eta \mapsto y(0,k,\eta)$ is $C^r$ for every $k\in \mathbb{Z}^+$.
\end{remark}

\section{Diffeomorphism of discrete topological equivalence under a dichotomy}

In this section, we study the differentiability properties of the function $\eta\mapsto w^*(0;(m,\eta)),$ in order to obtain that the topological equivalence is of class $C^1.$

\begin{lemma}\label{170}
If conditions \textbf{(d0)-(d7)} hold, then $\eta\mapsto w^*(0;(m,\eta))$ is a $C^1$ map.
\end{lemma}

\begin{proof} Fix $\eta\in \mathbb{R}^d$ and let $(\delta_n)_{n\in \mathbb{N}}\subset \mathbb{R}^d$ be a properly convergent to zero sequence. Choose and fix $m\in \mathbb{Z}^+$, we define
$$\psi_n(j)=\mathcal{G}(0,j+1)\frac{f(j,y(j,m,\eta+\delta_n))-f(j,y(j,m,\eta))-\frac{\partial f}{\partial u}(j,y(j,m,\eta))\frac{\partial y}{\partial \eta}(j,m,\eta)\delta_n}{|\delta_n|},$$
note that as $\eta\mapsto y(j,m,\eta)$ is continuous, then $\lim_{n\to \infty}y(j,m,\eta+\delta_n)=y(j,m,\eta)$. Applying \textbf{(d6)}, we have:
$$\lim_{n\to\infty}\psi_n(j)=0.$$

In the proof of Theorem \ref{116}, in \cite{Castaneda4}, authors define for $j<m$:
\begin{equation*}
\mathcal{C}_m(j)= \mathlarger{\prod_{i=j}^{m-1}}\frac{\norm{A^{-1}(i)}}{1-\norm{A^{-1}(i)\gamma(i)}}.
\end{equation*}

Thus, for $j<m,$ it is proved that
\begin{equation*}
|y(j,m,\eta)-y(j,m,\Tilde{\eta})| \leq \mathcal{C}_m(j)|\eta-\Tilde{\eta}|.
\end{equation*}

On the other hand, we know
\begin{eqnarray*}
    y(m+1,m,\eta)=A(m)\eta+f(m,\eta),
\end{eqnarray*}
hence
\begin{eqnarray*}
    y(m+1,m,\eta)-y(m+1,m,\Tilde{\eta})=A(m)(\eta-\Tilde{\eta})+f(m,\eta)-f(m,\Tilde{\eta}).
\end{eqnarray*}

So, we conclude
\begin{eqnarray}\label{173}
    |y(m+1,m,\eta)-y(m+1,m,\Tilde{\eta})|&\leq& \norm{A(m)}|\eta-\Tilde{\eta}|+|f(m,\eta)-f(m,\Tilde{\eta})| \nonumber \\
    \\
    &\leq &\left(\norm{A(m)}+\gamma(m)\right)|\eta-\Tilde{\eta}|.\nonumber
\end{eqnarray}

For $j> m$, we define
\begin{equation}\label{192}
\mathcal{B}_m(j):=\mathlarger{\prod_{i=m}^{j-1}}\norm{A(i)}+\gamma(i),
\end{equation}
which, by (\ref{173}), allows us to write for $j\geq m$
\begin{equation*}
|y(j,m,\eta)-y(j,m,\Tilde{\eta})| \leq \mathcal{B}_m(j)|\eta-\Tilde{\eta}|.
\end{equation*}

Now, we define
 $$\mathcal{A}_m(j):=\left\{ \begin{array}{lc}
            \mathcal{C}_m(j) & j<m \\
            \\1 & j=m\\
            \\ \mathcal{B}_m(j) & j> m,
             \end{array}
   \right.$$
hence, for every $j\in \mathbb{Z}^+$
\begin{equation*}
|y(j,m,\eta)-y(j,m,\Tilde{\eta})| \leq \mathcal{A}_m(j)|\eta-\Tilde{\eta}|.
\end{equation*}

So, by continuity of the norm
\begin{eqnarray*}
    \norm{\frac{\partial y}{\partial \eta}(j,m,\eta)}&=&\lim_{\delta\to 0}\frac{|y(j,m,\eta+\delta)-y(j,m,\eta)|}{|\delta|}\\
    \\
    &\leq&\lim_{\delta\to 0}\mathcal{A}_m(j)=\mathcal{A}_m(j).
\end{eqnarray*}

Similarly
\begin{eqnarray*}
    \norm{\frac{\partial f}{\partial u}(j,u)}&=&\lim_{\delta\to 0}\frac{|f(j,u+\delta)-f(j,u)|}{|\delta|}\\
    \\
    &\leq&\lim_{\delta\to 0}\gamma(j)=\gamma(j).
\end{eqnarray*}

Hence
\begin{eqnarray*}
    |\psi_n(j)|&\leq& \norm{\mathcal{G}(0,j+1)}\frac{|f(j,y(j,m,\eta+\delta_n))-f(j,y(j,m,\eta))|+\left|\frac{\partial f}{\partial u}(j,y(j,m,\eta))\frac{\partial y}{\partial \eta}(j,m,\eta)\delta_n\right|}{|\delta_n|}\\
    \\
    &\leq& \norm{\mathcal{G}(0,j+1)}\frac{\gamma(j)|y(j,m,\eta+\delta_n)-y(j,m,\eta)|+\gamma(j)\norm{\frac{\partial y}{\partial \eta}(j,m,\eta)}|\delta_n|}{|\delta_n|}\\
    \\
    &\leq& \norm{\mathcal{G}(0,j+1)}\gamma(j)\left(\frac{|y(j,m,\eta+\delta_n)-y(j,m,\eta)|}{|\delta_n|}+\norm{\frac{\partial y}{\partial \eta}(j,m,\eta)}\right)\\
    \\
    &\leq& 2\norm{\mathcal{G}(0,j+1)}\gamma(j)\mathcal{A}_m(j).
\end{eqnarray*}

On the other hand, by using \textbf{(d4)} and \textbf{(d7)} we have that
\begin{eqnarray*}
\sum_{j=0}^{\infty}\norm{\mathcal{G}(0,j+1)}\gamma(j) \mathcal{A}_m(j) &\leq& \sum_{j=0}^{m-1}\norm{\mathcal{G}(0,j+1)}\gamma(j) \mathcal{C}_m(j) +\sum_{j=m}^{\infty}\norm{\mathcal{G}(0,j+1)}\gamma(j) \mathcal{B}_m(j) \\
\\
&\leq &\sum_{j=0}^{m-1}\norm{\mathcal{G}(0,j+1)}\gamma(j) \max_{i<m}\mathcal{C}_m(i) +\sum_{j=m}^{\infty}D(j+1)h(j+1)\gamma(j) \mathcal{B}_m(j) \\
\\
&\leq &q\max_{i<m}\mathcal{C}_m(i) + \sum_{j=m}^{\infty}D(j+1)h(j+1)\gamma(j) \mathcal{B}_m(j)<+\infty.
\end{eqnarray*}

So finally, using Lebesgue's dominated convergence theorem we obtain
$$
\begin{array}{ccc}
&\displaystyle\lim_{n\to \infty}\frac{ w^*(0;(m,\eta+\delta_n))-w^*(0;(m,\eta))+\left[ \sum_{j=0}^\infty \mathcal{G}(0,j+1)\frac{\partial f}{\partial u}(j,y(j,m,\eta))\frac{\partial y}{\partial \eta}(j,m,\eta)\right]\delta_n}{|\delta_n|}&\\\\
& =-
\lim_{n\to \infty}\sum_{j=0}^\infty\psi_n(j)&\\\\
&=-\sum_{j=0}^\infty \left(\lim_{n\to\infty}\psi_n(j)\right)=0,&
\end{array}
$$
which implies $\eta\mapsto w^*(0;(m,\eta))$ is differentiable.
\end{proof}

\begin{theorem}\label{174}
If conditions \textbf{(d0)-(d7)} hold, and \textbf{(d6)} is satisfied with $r=1$, then (\ref{98}) and (\ref{99}) are $C^1$-topologically equivalent on $\mathbb{Z}^+$.
\end{theorem}

\begin{proof} By Theorem \ref{116} we know they are topologically equivalent. By Lemma \ref{170} $\eta\mapsto w^*(0;(k,\eta))$ is $C^1$ differentiable for every $k\in \mathbb{Z}^+$,  and, as stated in Remark \ref{191}, $\eta \mapsto y(0,k,\eta)$ is as well. By using (\ref{107}) allows us to conclude $\eta\mapsto G(k,\eta)$ is $C^1$ differentiable.\\

Furthermore, as $G$ is a topological equivalence, we know $\xi \mapsto G(k,\xi)-\xi$ is bounded and so $G(k,\xi)\to \infty$ when $|\xi|\to \infty$. This, combined with the previous expression, implies by Corollary 2.1 of \cite{Plastock}, that $\xi \mapsto G(k,\xi)$ is a diffeomorphism. Moreover, as $G(k,H(k,\xi))=\xi$, we have
$$\frac{\partial G}{\partial \xi}(k,H(k,\xi))\frac{\partial H}{\partial \xi}(k,\xi)=I,$$
and so $\frac{\partial H}{\partial \xi}(k,\xi)=\left[ \frac{\partial G}{\partial \xi}(k,H(k,\xi))\right]^{-1}$, which completes the proof.
\end{proof}

\begin{corollary}\label{175}
Suppose systems (\ref{98}) and (\ref{99}) satisfy \textbf{(d0)} and \textbf{(d6)} with $r=1$. Moreover, suppose \textbf{(d1)} is fulfilled with a classical exponential dichotomy, which means, $P$ and $Q$ are complementary constant projectors, $D(n)=D>0$ for every $n\in \mathbb{Z}^+$ and $h(n)=\theta^n$, with $\theta\in (0,1)$, and \textbf{(d2)} is fulfilled with $\gamma(n)=\gamma>0$ and $\mu(n)=\mu>0$ for every $n \in \mathbb{Z}^+$. Then, if $M\gamma<1$, $\frac{D \gamma}{1-\theta}<1$ and $\theta(M+ \gamma)<1$, the systems (\ref{98}) and (\ref{99}) are $C^1$-topologically equivalent on $\mathbb{Z}^+$.
\end{corollary}

\begin{proof} Note that for an arbitrary $n\in \mathbb{Z}^+$ we have:
$$\norm{A^{-1}(n)\gamma}\leq M\gamma <1,$$
hence \textbf{(d5)} is satisfied. On the other hand, for an arbitrary $k\in \mathbb{Z}^+$ we have
\begin{eqnarray*}
    \sum_{j=0}^\infty \norm{\mathcal{G}(k,j+1)}\mu&=&\sum_{j=0}^{k-1} \norm{\mathcal{G}(k,j+1)}\mu+\sum_{j=k}^\infty \norm{\mathcal{G}(k,j+1)}\mu\\
    \\
    &=&\sum_{j=0}^{k-1} \norm{\Phi(k,j+1)P}\mu+\sum_{j=k}^\infty \norm{\Phi(k,j+1)Q}\mu\\
    \\
    &\leq&\sum_{j=0}^{k-1} \frac{\theta^{k}}{\theta^{j+1}}D\mu+\sum_{j=k}^\infty \frac{\theta^{j+1}}{\theta^{k}}D\mu\\
    \\
    &=&\theta^{k-1}D \mu \frac{1-\left(\frac{1}{\theta}\right)^{k}}{1-\frac{1}{\theta}}+\theta D\mu\frac{1}{1-\theta}\\
    \\
    &=&D \mu \frac{1+\theta-\theta^{k}}{1-\theta }\leq \frac{D\mu}{1-\theta}.\\
\end{eqnarray*}

Thus \textbf{(d3)} is satisfied. Analogously
$$\sum_{j=0}^\infty \norm{\mathcal{G}(k,j+1)}\gamma\leq \frac{D \gamma}{1-\theta}<1.$$

Hence \textbf{(d4)} is satisfied. Finally, for an arbitrary $m \in \mathbb{Z}^+$ we have

\begin{eqnarray*}
    \mathlarger{\sum_{j=m}^{\infty}}D \gamma\left(h(j+1)\left[ \mathlarger{\prod_{p=m}^{j-1}}\norm{A(p)}+\gamma\right]\right)&\leq&\mathlarger{\sum_{j=m}^{\infty}}D \gamma\left(\theta^{j+1}\left[ \mathlarger{\prod_{p=m}^{j-1}}M+\gamma\right]\right)\\
    \\
    &\leq& D \gamma \mathlarger{\sum_{j=m}^{\infty}}\theta^{j+1} (M+\gamma)^{j-m}\\
    \\
    &\leq&\frac{D \gamma\theta}{(M+\gamma)^m} \mathlarger{\sum_{j=m}^{\infty}}\left(\theta(M+\gamma)\right)^{j}<\infty ,\\
\end{eqnarray*}
hence \textbf{(d7)} is satisfied. So, by applying Theorem \ref{174}, the systems are $C^1$-topologically equivalent on $\mathbb{Z}^+$.
\end{proof}

\begin{remark}\label{186}
The conditions imposed in  corollary \ref{175} are easily reachable with $\theta$ being small enough. For example, set $\theta=0.1$, $M=2$ and $\gamma=0.4$.
\end{remark}

\begin{corollary}\label{176}
Suppose conditions \textbf{(d0)} and \textbf{(d5)} are satisfied. Also, consider system (\ref{98}) admits a nonuniform exponential dichotomy, which means there are two complementary invariant projectors $P(\cdot)$ and $Q(\cdot)$ and constants $C,\lambda, \varepsilon>0$ such that:
$$\left\{ \begin{array}{lc}
            \norm{\Phi(k,n)P(n)}\leq C
            e^{-\lambda(k-n)+\varepsilon n}, &\forall k \geq n\geq 0 \\
            \\ \norm{\Phi(k,n)Q(n)}\leq C
            e^{\lambda(k-n)+\varepsilon n}, &\forall 0\leq k\leq n.
             \end{array}
   \right.$$
   
Furthermore, suppose that for every $k\in \mathbb{Z}^+$ $u\mapsto f(k,u)$ is a $C^1$ function such that $u\mapsto \frac{\partial f}{\partial u}(k,u)$ is a bounded map that satisfies:
\begin{equation}\label{183}
    \norm{\frac{\partial f}{\partial u}(k,u)}\leq \nu e^{-\varepsilon (k+1)}
\end{equation}
and 
\begin{equation}\label{184}
    |f(k,u)|\leq \kappa e^{-\tau (k+1)}
\end{equation}

for given $\nu ,\kappa >0$ and $\tau > \varepsilon - \lambda$. Then, if $Me^{-\lambda}<1$, for sufficiently small $\nu>0$, systems (\ref{98}) and (\ref{99}) are $C^1$-topologically equivalent on $\mathbb{Z}^+.$
\end{corollary}

\begin{proof} Condition \textbf{(d1)} is easily verified with $D(n)=Ce^{\varepsilon n}$ and $h(n)=e^{-\lambda n}$. Note that using generalized intermediate point Theorem, condition (\ref{183}) implies:
$$|f(k,y)-f(k,\Tilde{y})|\leq \nu e^{-\varepsilon (k+1)}|y-\Tilde{y}|,$$
so condition \textbf{(d2)} is verified with $\gamma(k)=\nu e^{-\varepsilon(k+1)}$ and $\mu(k)=\kappa e^{-\tau(k+1)}$. Hence condition \textbf{(d3)} is immediately satisfied and, provided $\nu$ is sufficiently small, so is \textbf{(d4)} in a similar fashion as in the previous Corollary. Condition \textbf{(d6)} is immediately satisfied by hypothesis, with $r=1$. Now, denote
$$\Psi_m(j)=\mathlarger{\prod_{p=m}^{j-1}}\norm{A(p)}+\gamma(p).$$

It is easy to see $m\geq n$ implies $\Psi_m(j)\leq \Psi_n(j)$, hence for a fixed $m\in \mathbb{Z}^+$:

$$\Psi_m(j)\leq \Psi_0(j)\leq \left(M+\nu e^{-\varepsilon} \right)^j$$

So we have
\begin{eqnarray*}
    \sum_{j=m}^\infty D(j+1)h(j+1)\gamma(j)\Psi_m(j)&\leq& \sum_{j=0}^\infty D(j+1)h(j+1)\gamma(j)\Psi_0(j)\\
    \\
    &\leq&C\nu e^{-\lambda} \sum_{j=0}^\infty\left[ e^{-\lambda}\left( M+\nu e^{-\varepsilon}\right)\right]^j,
\end{eqnarray*}
which, as $Me^{-\lambda}<1$, is finite provided $\nu$ is sufficiently small, so condition \textbf{(d7)} is satisfied. Applying Theorem \ref{174}, the corollary is proved.
\end{proof}

\begin{remark}
Note that for condition (\ref{183}) to be satisfied, we only have $\tau$ to be bigger than $\varepsilon-\lambda$, regardless of the sign of such $\tau$. This could mean that, in the case  $\lambda>\varepsilon$, this dichotomy even admits a non bounded perturbation with exponential growth, it just has to have a increasing ratio smaller than the decay ratio of the dichotomy.
\end{remark}

\begin{remark}
Corollary \ref{176} has been inspired by \cite[Theorem 2]{Dragicevic}. Our approach and that result consider conditions \textbf{(d0)} and (\ref{183}), the same dichotomy and both results work for $\nu$ sufficiently small. Furthermore, we imposed the extra hypothesis $Me^{-\lambda}<1$ and conditions (\ref{184}) and \textbf{(d5)}. However, as stated  on Remark \ref{185}, we think this last hypothesis may be dropped.

Moreover, our result gives $C^1$-topologically equivalence on $\mathbb{Z}^+$, in contrast to the mentioned theorem, which gives it on the whole $\mathbb{Z}$. Nevertheless, our approach does not rely on spectral properties of the dichotomy, which is the core of that result.
\end{remark}

\begin{example}\label{187}
Let $(a_n)_{n\in \mathbb{N}},(b_n)_{n\in \mathbb{N}},(c_n)_{n\in \mathbb{N}}$ three sequences such that
\begin{equation}\label{179}
    0<\alpha \leq a_n,b_n\leq c_n^{-1}\leq 1 \leq c_n \leq M,
\end{equation}
for given $M,\alpha >0$, and let $(c_n)_{n\in \mathbb{N}}$ be monotone increasing. Consider $A(n)\in \mathcal{M}_{3\times 3}(\mathbb{R})$ the diagonal matrix given by
\begin{equation}\label{178}
A(n)=\begin{pmatrix}
a_n & 0 & 0\\
0 & b_n & 0\\
0&0&c_n
\end{pmatrix},
\end{equation}
and consider the system (\ref{98}) associated to this sequence of matrices. Let $(r_n)_{n\in \mathbb{N}}$ be a sumable non negative sequence and define $(\gamma_n)_{n\in \mathbb{N}}$ the sequence given by
$$\gamma_n=\frac{r_nc_n}{r_{n-1}+\dots + r_1+1}.$$

Let $g:\mathbb{R}^3\to \mathbb{R}^3$ be a bounded, differentiable Lipschitz (with constant $\leq 1$) map, define $f:\mathbb{Z}^+\times \mathbb{R}^3\to \mathbb{R}^3$ given by $f(k,x)=\gamma_kg(x)$ and consider the the system (\ref{99}) associated to this perturbation and the previous linear system. Then, if $\sup_{k\in \mathbb{Z}^+}\frac{c_{k-1}r_{k-1}}{r_{k-2}+r_{k-3}+\cdots+1}+\sum_{j=1, j\neq k-1}^\infty r_j<1$, systems (\ref{98}) and (\ref{99}) are $C^1$-topologically equivalent on $\mathbb{Z}^+$.
\end{example}

In fact, since $0\leq a_n,b_n\leq c_n$, then $\norm{A(n)}=c_n\leq M$ and $\norm{A^{-1}(n)}\leq \alpha^{-1}$, hence \textbf{(d0)} is fulfilled. Furthermore, let $P(\cdot)$ and $Q(\cdot)$ be
\begin{equation}\label{177}
    P(n)=\begin{pmatrix}
1 & 0 & 0\\
0 & 1 & 0\\
0&0&0
\end{pmatrix} \text{ and }Q(n)=\begin{pmatrix}
0 & 0 & 0\\
0 & 0 & 0\\
0&0&1
\end{pmatrix}.
\end{equation}

Then, if we consider $\Phi(m,n)$ be the transition matrix for (\ref{98}), it is clear that
\begin{equation}\label{182}
    \left\{ \begin{array}{lc}
            \norm{\Phi(k,n)P(n)}\leq \prod_{j=n}^{k-1}\max \{ a_j,b_j\}, &\forall k \geq n\geq 0 \\
            \\ \norm{\Phi(k,n)Q(n)}\leq \prod_{j=k}^{n-1}c_j^{-1} , &\forall 0\leq k\leq n.
             \end{array}
   \right.
\end{equation}

Hence, setting $h(n)=\prod_{p=1}^{n-1}c_p^{-1}$ and $D(n)=1$, \textbf{(d1)} follows. Note that \textbf{(d2)} and \textbf{(d6)} follow immediately by hypothesis. Moreover

$$\norm{A(n)^{-1}\gamma_n}\leq \alpha^{-1} \frac{r_nc_n}{r_{n-1}+\dots +r_1+1}\leq M\alpha^{-1}r_n.$$

As $r_n\to 0$, we obtain \textbf{(d5)}. It is easy to inductively prove that
$$\gamma_j=r_jc_j\left( \prod_{p=1}^{j-1}1+\frac{\gamma_p}{c_p}\right)^{-1}=r_jc_j \prod_{p=1}^{j-1}\frac{c_p}{\gamma_p+c_p}.$$

Now, denote
$$\Psi_m(j)=\mathlarger{\prod_{p=m}^{j-1}}\norm{A(p)}+\gamma_p.$$

It is easy to see $m\geq n$ implies $\Psi_m(j)\leq \Psi_n(j)$, hence for a fixed $m\in \mathbb{Z}^+$
$$\Psi_m(j)\leq \Psi_1(j)\leq \prod_{p=1}^{j-1}c_p+\gamma_p.$$

So we have
\begin{eqnarray*}
    \sum_{j=m}^\infty D(j+1)h(j+1)\gamma(j)\Psi_m(j)&\leq& \sum_{j=1}^\infty h(j+1)\gamma(j)\Psi_1(j)\\
    \\
    &\leq& \sum_{j=1}^\infty \gamma(j)\prod_{p=1}^{j}c_p^{-1}\prod_{p=1}^{j-1}(c_p+\gamma_p)\\
    \\
    &\leq& \sum_{j=1}^\infty \gamma_jc_j^{-1} \left(\prod_{p=1}^{j-1}1+\frac{\gamma_p}{c_p} \right)\\
    \\
    &=&\norm{r}_1<\infty.
\end{eqnarray*}

Hence \textbf{(d7)} is fulfilled. Finally note that for a fixed $k\in \mathbb{Z}^+$
\begin{eqnarray*}
    \sum_{j=0}^\infty \norm{\mathcal{G}(k,j+1)}\gamma_j&\leq& \sum_{j=0}^{k-1}\frac{h(k)}{h(j+1)} \gamma_j+\sum_{j=k}^\infty\frac{h(j+1)}{h(k)} \gamma_j    \\
    \\
    &\leq& \sum_{j=1}^{k-1}\left(\prod_{p=j+1}^{k-1}c_p^{-1}\right)\frac{c_jr_j}{r_{j-1}+\cdots r_1+1}\\
    \\
    &&+\sum_{j=k}^\infty r_jc_j \left(\prod_{p=k}^{j}c_p^{-1}\right)\left(\prod_{p=1}^{j-1} 1+\frac{\gamma_p}{c_p}\right)^{-1}\\
    \\
    &\leq& \sum_{j=1}^{k-1}\left(\prod_{p=j+1}^{k-1}c_p^{-1}\right)\frac{c_{k-1}r_j}{r_{j-1}+\cdots r_1+1}\\
    &&+\sum_{j=k}^\infty r_j  \left(\prod_{p=0}^{k-1}1+\frac{\gamma_p}{c_p} \right)^{-1}\left(\prod_{p=k}^{j-1}\frac{1}{c_p+\gamma_p} \right)\\
    &<&\frac{c_{k-1}r_{k-1}}{r_{k-2}+r_{k-3}+\cdots+1}+\sum_{j=1, j\neq k-1}^\infty r_j.
\end{eqnarray*}

Hence it satisfies \textbf{(d4)} and \textbf{(d3)} with $\mu_j=\gamma_j \norm{g}_\infty$. Applying Theorem \ref{174}, we show our claim.

\begin{remark}
Condition $\sup_{k\in \mathbb{Z}^+}\frac{c_{k-1}r_{k-1}}{r_{k-2}+r_{k-3}+\cdots+1}+\sum_{j=1, j\neq k-1}^\infty r_j<1$ is obtained, for example, if the increasing rate of $(c_n)$ is smaller than the growing rate of the partial sums of $(r_n)$ and $\norm{r}_1<1$. Another simpler case is if $M\norm{r}_1<1$.
\end{remark}

\begin{example}\label{188}
Let $(a_n)_{n\in \mathbb{N}},(b_n)_{n\in \mathbb{N}},(c_n)_{n\in \mathbb{N}}$ three sequences as in (\ref{179}), $A(n)$ as in (\ref{178}) and the system (\ref{98}) associated to this sequence of matrices. Let $(r_n)_{n\in \mathbb{N}}$ be a summable non negative sequence and define $(\gamma_n)_{n\in \mathbb{N}}$ the sequence given by:
$$\gamma_n=\frac{r_n}{r_{n-1}+\dots +r_1+1}.$$

Let $g:\mathbb{R}^3\to \mathbb{R}^3$ be a bounded, differentiable Lipschitz (with constant $\leq 1$) map, define $f:\mathbb{Z}^+\times \mathbb{R}^3\to \mathbb{R}^3$ given by $f(k,x)=\gamma_kg(x)$ and consider the the system (\ref{99}) associated to this perturbation and the previous linear system. Then, if $\norm{r}_1<1$, systems (\ref{98}) and (\ref{99}) are $C^1$-topologically equivalent on $\mathbb{Z}^+$.
\end{example}

Indeed, conditions \textbf{(d0)}, \textbf{(d1)}, \textbf{(d2)}, \textbf{(d5)} and  \textbf{(d6)} follow easily in the same fashion as in the previous Example, with $P(\cdot)$ and $Q(\cdot)$ as in (\ref{177}). It is easy to inductively prove
$$\gamma_j=r_j\left( \prod_{p=1}^{j-1}1+\gamma_p\right)^{-1}.$$

Denote
$$\Psi_m(j)=\mathlarger{\prod_{p=m}^{j-1}}\norm{A(p)}+\gamma_p.$$

It is easy to see $m\geq n$ implies $\Psi_m(j)\leq \Psi_n(j)$, hence for a fixed $m\in \mathbb{Z}^+$
$$\Psi_m(j)\leq \Psi_1(j)\leq \prod_{p=1}^{j-1}c_p+\gamma_p.$$

So we have
\begin{eqnarray*}
    \sum_{j=m}^\infty D(j+1)h(j+1)\gamma(j)\Psi_m(j)&\leq& \sum_{j=1}^\infty h(j+1)\gamma(j)\Psi_1(j)\\
    \\
    &\leq& \sum_{j=1}^\infty r_j\left(\prod_{p=1}^{j-1}1+\gamma_p \right)^{-1}\left(\prod_{p=1}^{j}c_p^{-1}\right)\left(\prod_{p=1}^{j-1}c_p+\gamma_p \right)\\
    \\
    &\leq& \sum_{j=1}^\infty r_jc_j^{-1} \left(\prod_{p=1}^{j-1}\frac{1+\gamma_pc_p^{-1}}{1+\gamma_p} \right)\\
    \\
    &\leq&\norm{r}_1<\infty.
\end{eqnarray*}

Hence \textbf{(d7)} is fulfilled. Finally note that for a fixed $k\in \mathbb{Z}^+$
\begin{eqnarray*}
    \sum_{j=0}^\infty \norm{\mathcal{G}(k,j+1)}\gamma_j&\leq& \sum_{j=1}^{k-1}\frac{h(k)}{h(j+1)} \gamma_j+\sum_{j=k}^\infty\frac{h(j+1)}{h(k)} \gamma_j    \\
    \\
    &\leq& \sum_{j=1}^{k-1}\left(\prod_{p=j+1}^{k-1}c_p^{-1}\right)\frac{r_j}{r_{j-1}+\cdots r_1+1}\\
    \\
    &&+\sum_{j=k}^\infty \left(\prod_{p=k}^{j}c_p^{-1}\right)\frac{r_j}{r_{j-1}+\cdots r_1+1}\\
    \\
    &<&\norm{r}_1<1,
\end{eqnarray*}
hence proving \textbf{(d4)} and \textbf{(d3)} with $\mu_j=\gamma_j \norm{g}_\infty$. Applying theorem \ref{174} our claim is proved.

\begin{example}\label{189}
Let $(a_n)_{n\in \mathbb{N}},(b_n)_{n\in \mathbb{N}},(c_n)_{n\in \mathbb{N}}$ three sequences as in (\ref{179}) and $A(n)$ as in (\ref{178}). Let $(E(n))_{n\in \mathbb{N}}$ be a sequence of uniformly bounded invertible matrices with uniformly bounded inverses. Set $\mathscr{E}^{+}$ to be a uniform bound of $E(n)$ and $\mathscr{E}^{-}$ to be a uniform bound for their inverses. Set $E(0)=I$ and define $B(n)=E^{-1}(n)A(n)E(n-1)$. Let $(\gamma_n)_{n\in \mathbb{N}}$ be a summable non negative sequence.\\

Let $g:\mathbb{R}^3\to \mathbb{R}^3$ be a bounded, differentiable Lipschitz (with constant $\leq 1$) map, define $f:\mathbb{Z}^+\times \mathbb{R}^3\to \mathbb{R}^3$ given by $f(k,x)=\gamma_kg(x)$. Consider the systems:
\begin{equation}\label{180}
    x(k+1)=B(k)x(k)
\end{equation}
\begin{equation}\label{181}
    y(k+1)=B(k)y(k)+f(k,y(k)).
\end{equation}

Finally, suppose there is a $\delta >0$ such that:
\begin{equation}\label{190}
  \sum_{j=1}^{\infty}\gamma_j\left( \mathscr{E}^-\mathscr{E}^{+}+\delta\right)^j<\infty.
\end{equation}

Then, if $\mathscr{E}^-\mathscr{E}^+\norm{\gamma}_1<1$, systems (\ref{180}) and (\ref{181}) are $C^1$-topologically equivalent on $\mathbb{Z}^+$.
\end{example}

In fact, whitout lost of generality, $\mathscr{E}^{-}, \mathscr{E}^{+}\geq 1$. Since $0\leq \alpha \leq a_n,b_n\leq c_n\leq M$, then $\norm{B(n)}\leq \mathscr{E}^{-} M \mathscr{E}^{+}$ and $\norm{B^{-1}(n)}\leq \mathscr{E}^{-} \alpha^{-1} \mathscr{E}^{+}$, hence \textbf{(d0)} is fulfilled.\\

Define $\Tilde{P}(\cdot)$ and $\Tilde{Q}(\cdot)$ as
$$\Tilde{P}(n)=E^{-1}(n-1)P(n)E(n-1) \text{ , and }\Tilde{Q}(n)=E^{-1}(n-1)Q(n)E(n-1),$$
where $P(\cdot)$ and $Q(\cdot)$ are as in (\ref{177}). Then, if we consider $\Phi(m,n)$ to be as in (\ref{182}), it is easy to see that
\begin{equation*}
\Tilde{\Phi}(k,n)=E^{-1}(k-1)\Phi(k,n)E(n-1),
\end{equation*}
where $\Tilde{\Phi}(k,n)$ is the transition matrix for (\ref{180}). As in the previous example it is immediate that
$$\left\{ \begin{array}{lc}
            \norm{\Tilde{\Phi}(k,n)\Tilde{P
    }(n)}\leq \mathscr{E}^{-}\norm{E(n-1)}\prod_{j=n}^{k-1}\max \{ a_j,b_j\}, &\forall k \geq n\geq 0 \\
            \\ \norm{\Tilde{\Phi}(k,n)\Tilde{Q}(n)}\leq \mathscr{E}^{-}\norm{E(n-1)}\prod_{j=k}^{n-1}c_j^{-1} , &\forall 0\leq k\leq n,
             \end{array}
   \right.$$
hence, setting $h(n)=\prod_{p=1}^{n-1}c_p^{-1}$ and $D(n)=\mathscr{E}^{-}\norm{E(n-1)}$, \textbf{(d1)} follows. Note that \textbf{(d2)} and \textbf{(d6)} follow immediately from our hypothesis, while \textbf{(d5)} is satisfied in the same fashion as in Example \ref{187}.\\

Denote
$$\Tilde{\Psi}_m(j)=\mathlarger{\prod_{p=m}^{j-1}}\norm{B(p)}+\gamma_p.$$

It is easy to see $m\geq n$ implies $\Tilde{\Psi}_m(j)\leq \Tilde{\Psi}_n(j)$, hence for a fixed $m\in \mathbb{Z}^+$:
$$\Tilde{\Psi}_m(j)\leq \Tilde{\Psi}_1(j)\leq \prod_{p=1}^{j-1} \mathscr{E}^{-}\mathscr{E}^{+}c_p+\gamma_p.$$

So we have
\begin{eqnarray*}
    \sum_{j=m}^\infty D(j+1)h(j+1)\gamma(j)\Tilde{\Psi}_m(j)&\leq& \sum_{j=1}^\infty D(j+1)h(j+1)\gamma(j)\Tilde{\Psi}_1(j)\\
    \\
    &\leq& \sum_{j=1}^\infty \mathscr{E}^{-}\norm{E(j)}\left(\prod_{p=1}^{j}c_p^{-1}\right)\gamma_j\left(\prod_{p=1}^{j-1}\mathscr{E}^{-}\mathscr{E}^{+}c_p+\gamma_p \right)\\
    \\
    &\leq& \mathscr{E}^{-}\mathscr{E}^{+}\sum_{j=1}^\infty \gamma_jc_j^{-1} \left(\prod_{p=1}^{j-1}\mathscr{E}^{-}\mathscr{E}^{+}+\gamma_pc_p^{-1} \right) \\
    \\
    &\leq& \mathscr{E}^{-}\mathscr{E}^{+}\sum_{j=1}^\infty \gamma_j \left(\prod_{p=1}^{j-1}\mathscr{E}^{-}\mathscr{E}^{+}+\gamma_p \right).
\end{eqnarray*}

As $\gamma_p \xrightarrow[p\to \infty]{}0$, there exist $p_0$ such that $\gamma_p\leq \delta$ for every $p \geq p_0$. Hence, applying condition (\ref{190}), it is easy to see \textbf{(d7)} is fulfilled. Finally

\begin{eqnarray*}
    \sum_{j=0}^\infty \norm{\mathcal{G}(k,j+1)}\gamma_j&\leq& \mathscr{E}^{-}\sum_{j=1}^{k-1}\frac{h(k)}{h(j+1)} \norm{E(j-1)}\gamma_j+\mathscr{E}^{-}\sum_{j=k}^\infty\frac{h(j+1)}{h(k)} \norm{E(j-1)}\gamma_j    \\
    \\
    &\leq& \mathscr{E}^{-}\mathscr{E}^{+}\sum_{j=1}^{k-1}\left(\prod_{p=j+1}^{k-1}c_p^{-1}\right)\gamma_j\\
    \\
    &&+\mathscr{E}^{-}\mathscr{E}^{+}\sum_{j=k}^\infty \left(\prod_{p=k}^{j}c_p^{-1}\right)\gamma_j\\
    \\
    &<&\mathscr{E}^{-}\mathscr{E}^{+}\norm{\gamma}_1<1.
\end{eqnarray*}

Hence proving \textbf{(d4)} and \textbf{(d3)} with $\mu_j=\gamma_j \norm{g}_\infty$. By using Theorem \ref{174} our claim follows.

\begin{remark}
As we mentioned in the introduction, conditions \textbf{(d0)-(d6)} have been adopted several times by many authors in order to study the topological equivalence between systems (\ref{98}) and (\ref{99}).  Hence, the novelty of this work is the condition \textbf{(d7)} and its implications.\\ 

Namely, in Corollary \ref{175} the fact that this condition is satisfied relies mainly (as explained on Remark \ref{186}) on the properties of the exponential dichotomy, which gives a evolution ratio for the solutions that allows \textbf{(d7)} to be fulfilled.  Meanwhile, in Examples \ref{187} and \ref{188}, this condition is achieved by a perturbation that is small enough and has a fast enough decay ratio, and actually we completely ignore the help that the dichotomy may have presented. Finally, in Corollary \ref{176} and Example \ref{189}, the condition \textbf{(d7)} is achieved through a combination of the contributions of both the perturbation and the dichotomy.

This shows that there are many possibilities to achieve this condition and hence our main result,  Theorem \ref{174}, is applicable in many different situations.

\end{remark}

\section{Higher order derivatives}

In order to study higher order derivatives of the homeomorphism of topological equivalence, we once again consider the expression (\ref{107}). Taking in account Remark \ref{191}, the map $\eta\mapsto y(0,m,\eta)$ is of class $C^r$ ($r\geq 1$) for every fixed $m \in \mathbb{Z}^+$ if conditions \textbf{(d0)-(d6)} are satisfied; hence the class of differentiability of the homeomorphism relies on the differentiability of the map $\eta \mapsto w^*(0;(m,\eta))$.\\

In this section we attempt to find conditions that generalize \textbf{(d7)} in order to obtain higher order derivatives for the homeomorphism.

\begin{lemma}\label{193}
Suppose conditions \textbf{(d0)-(d7)} hold, where \textbf{(d6)} is fulfilled with $r=2$. Also, there are functions $\Gamma:\mathbb{Z}^+\to \mathbb{R}^+$ and $\pi_m:\mathbb{Z}^+\to \mathbb{R}^+$ such that: 
\begin{equation}\label{198}
    \norm{\frac{\partial^2 f}{\partial u^2}(j,u)}\leq \Gamma(j)\text{ , for every }j\in \mathbb{Z}^+
\end{equation}
and 
\begin{equation}\label{199}
    \norm{\frac{\partial^2 y}{\partial \eta^2}(j,m,\eta)}\leq \pi_m(j)\text{ , for every }j\geq m.
\end{equation}

Finally, suppose that for each fixed $m\in \mathbb{Z}^+$, the sequences satisfy:
    \begin{equation}\label{200}
    \mathlarger{\sum_{j=m}^{\infty}}\left(D(j+1)h(j+1)\left\{\pi_m(j)\gamma(j)+\Gamma(j)\left[ \mathlarger{\prod_{p=m}^{j-1}}\norm{A(p)}+\gamma(p)\right]^2 \right\}\right)<+\infty.
    \end{equation}
Then $\eta\mapsto w^*(0;(m,\eta))$ is a $C^2$ map.
\end{lemma}

\begin{proof} From the proof of Lemma \ref{170} it is easy to see that for every fixed $m\in \mathbb{Z}^+$
$$ \frac{\partial w^*(0;(m,\eta))}{\partial \eta}=-\sum_{j=0}^\infty \mathcal{G}(0,j+1)\frac{\partial f}{\partial u}(j,y(j,m,\eta))\frac{\partial y}{\partial \eta}(j,m,\eta).$$

Hence, in order to prove the differentiability of this map, it is enough to find a summable  function that uniformly (respect to $\eta$) dominates

$$\mathcal{G}(0,j+1)\left[\frac{\partial^2 y}{\partial \eta^2}(j,m,\eta)\frac{\partial  f}{\partial u}(j,y(j,m,\eta))+\frac{\partial ^2 f}{\partial u^2}(j,y(j,m,\eta))\left( \frac{\partial y}{\partial \eta}(j,m,\eta)\right)^2\right]$$

for every $j\geq 0$. 

This fact is verified by hypothesis. Therefore, we apply the same strategy as in the proof of Lemma \ref{170}.
\end{proof}

\begin{theorem}\label{194}
If conditions \textbf{(d0)-(d7)} hold, \textbf{(d6)} is satisfied with $r=2$, and conditions of Lemma \ref{193} are satisfied, then (\ref{98}) and (\ref{99}) are $C^2$-topologically equivalent on $\mathbb{Z}^+$.
\end{theorem}

Note that Theorem \ref{194} follows easily in the same fashion as the proof of Theorem \ref{174}.

\begin{corollary}
Suppose conditions \textbf{(d0)} and \textbf{(d5)} are satisfied. Also, consider system (\ref{98}) admits a nonuniform exponential dichotomy, which means there are two complementary invariant projectors $P(\cdot)$ and $Q(\cdot)$ and constants $C,\lambda, \varepsilon>0$ such that
$$\left\{ \begin{array}{lc}
            \norm{\Phi(k,n)P(n)}\leq C
            e^{-\lambda(k-n)+\varepsilon n}, &\forall k \geq n\geq 0 \\
            \\ \norm{\Phi(k,n)Q(n)}\leq C
            e^{\lambda(k-n)+\varepsilon n}, &\forall 0\leq k\leq n.
             \end{array}
   \right.$$
   
Furthermore, suppose that for every $k\in \mathbb{Z}^+$ $u\mapsto f(k,u)$ is a $C^2$ function such that:
\begin{equation*}
    |f(k,u)|\leq \kappa e^{-\tau (k+1)},
\end{equation*}

\begin{equation*}
    \norm{\frac{\partial f}{\partial u}(k,u)}\leq \nu e^{-\varepsilon (k+1)}
\end{equation*}
and 
\begin{equation}\label{197}
    \norm{\frac{\partial f}{\partial u}(k,x)-\frac{\partial f}{\partial u}(k,y)}\leq \zeta e^{-\varepsilon (k+1)}|x-y|,
\end{equation}

for given $\kappa,\nu ,\zeta>0$ and $\tau > \varepsilon - \lambda$. Then, if $M^2e^{-\lambda}<1$, for sufficiently small $\nu>0$, the systems (\ref{98}) and (\ref{99}) are $C^2$-topologically equivalent on $\mathbb{Z}^+.$
\end{corollary}

\begin{proof} By Corollary \ref{176} we know that conditions \textbf{(d0)-(d7)} are verified and hence they are $C^1$-topologically equivalent on $\mathbb{Z}^+$. In particular, \textbf{(d4)} is satisfied with $\gamma(j)=\nu e^{-\varepsilon(j+1)}$.\\ 

It is easy to see that (\ref{197}) implies (\ref{198}), with $\Gamma(j)=\zeta e^{-\varepsilon
(j+1)}$. In addition, if $M^2e^{-\lambda}<1$, this implies, in a similar fashion as in the proof of \ref{176}, that
    \begin{equation}\label{202}
    \mathlarger{\sum_{j=m}^{\infty}}\left(D(j+1)h(j+1)\Gamma(j)\left[ \mathlarger{\prod_{p=m}^{j-1}}\norm{A(p)}+\gamma(p)\right]^2 \right)<+\infty.
    \end{equation}

Fix $m\in \mathbb{Z}$. We know that $\eta \mapsto \frac{\partial y}{\partial \eta}(j,m,\eta)$ is well defined for every $j\geq m$. Consider the matrix initial value problem:
\begin{equation}\label{201}
    \left\{ \begin{array}{ccl}
            z(j+1) & = & \left[ A(j)+\frac{\partial f}{\partial u}(j,y(j,m,\eta)\right]z(j)\\
            \\z(m) & = & I.
\end{array}
\right.\end{equation}

Note that $j\mapsto z(j,m,\eta)=\frac{\partial y}{\partial \eta}(j,m,\eta)$ is a solution of (\ref{201}), hence
$$z(m+1,m,\eta)=\left[A(m)+\frac{\partial f}{\partial u}(m,y(m,m,\eta))\right] z(m,m,\eta)=A(m)+\frac{\partial f}{\partial u}(m,\eta).$$

Then, we have
$$\norm{z(m+1,m,\eta)-z(m+1,m,\Tilde{\eta})}=\norm{\frac{\partial f}{\partial u}(m,\eta)-\frac{\partial f}{\partial u}(m,\Tilde{\eta})}\leq \Gamma(m)|\eta-\Tilde{\eta}|.$$

So, inductively we can prove
$$\norm{z(j,m,\eta)-z(j,m,\Tilde{\eta})}\leq\left( \prod_{i=m}^{j-1}\Gamma(i)\right) |\eta-\Tilde{\eta}|.$$

Which implies that $\pi_m(j)=\prod_{i=m}^{j-1}\Gamma(i)$ satisfies condition (\ref{199}) from Lemma \ref{193}. As $\Gamma(j)\to 0$ when $j\to \infty$, it is easy to see there are constants $a_m,b>0$, $b<1$ such that $\pi_m(j)\leq  a_mb^j$, which allows us to easily see
    \begin{equation*}
    \mathlarger{\sum_{j=m}^{\infty}}\left(D(j+1)h(j+1)\pi_m(j)\gamma(j)\right)<+\infty.
    \end{equation*}
    
Now, condition (\ref{200}) from Lemma \ref{193} is satisfied if we consider (\ref{202}). Finally, applying Theorem \ref{194} we obtain the result.
\end{proof}

\section{Discussion}

Suppose that conditions \textbf{(d0)-(d6)} are verified. Suppose as well there are functions $\Gamma_s:\mathbb{Z}^+\to \mathbb{R}^+$, $\pi_{s,m}:\mathbb{Z}^+\to \mathbb{R}^+$, $1\leq s \leq r$, $m \in \mathbb{Z}^+$ such that: 
\begin{equation}\label{195}
    \norm{\frac{\partial^s f}{\partial u^s}(j,u)}\leq \Gamma_s(j)\text{ , for every }j\in \mathbb{Z}^+
\end{equation}
and 
\begin{equation}\label{196}
    \norm{\frac{\partial^s y}{\partial \eta^s}(j,m,\eta)}\leq \pi_{s,m}(j)\text{ , for every }j\geq m.
\end{equation}

Chose a fixed $m \in \mathbb{Z}^+$ and consider the set of functions
$$\mathfrak{S}_m=\left\{\Gamma_s\prod_{k=1}^r\pi_{k,m}^{e_k}: 1\leq s \leq r, e_k\in \mathbb{Z}^+_0 \right\}.$$

Consider the module $\mathbb{Z}[\mathfrak{S}_m]$ and define the $\mathbb{Z}-$linear map $\mathbb{D}_m:\mathbb{Z}[\mathfrak{S}_m]\to \mathbb{Z}[\mathfrak{S}_m]$ given by\\
\begin{itemize}
    \item $\mathbb{D}_m(\Gamma_s)=\Gamma_{s+1}\pi_{1,m}$, for every $1\leq s < r.$
    \item $\mathbb{D}_m(\pi_{s,m})=\pi_{s+1,m}$, for every $1\leq s < r.$
    \item $\mathbb{D}_m(fg)=\mathbb{D}_m(f)g+f\mathbb{D}_m(g)$, for every $f,g \in \mathbb{Z}[\mathfrak{S}_m]$ such that $fg\in \mathbb{Z}[\mathfrak{S}_m].$\\
\end{itemize}

In order to simplify notations, suppose the existence of a certain function $\Gamma_0$ such that $\mathbb{D}_m(\Gamma_0)(j)=\Gamma_1(j)\pi_{1,m}(j)$. Under these conditions it is easy to see:\\
\begin{itemize}
    \item $\mathbb{D}^2_m(\Gamma_0)(j)=\Gamma_2(j)\pi_{1,m}(j)^2+\Gamma_1(j)\pi_{2,m}(j).$
    \item $\mathbb{D}^3_m(\Gamma_0)(j)=\Gamma_3(j)\pi_{1,m}(j)^3+3\Gamma_2(j)\pi_{2,m}(j)\pi_{1,m}(j)+\Gamma_1(j)\pi_{3,m}(j).$
    \item $\mathbb{D}^4_m(\Gamma_0)(j)=\Gamma_4(j)\pi_{1,m}(j)^4+6\Gamma_3(j)\pi_{2,m}(j)\pi_{1,m}(j)^2+4\Gamma_2(j)\pi_{3,m}(j)\pi_{1,m}(j)+\Gamma_1(j)\pi_{4,m}(j)$.\\
\end{itemize}

And inductively it is easy to calculate higher powers of $\mathbb{D}_m$ applied to $\Gamma_0$. With this in mind, we would like to introduce the following condition:\\

\textbf{(DIF,r)} Suppose conditions \textbf{(d0)-(d6)} are verified. Suppose as well there are functions $\Gamma_s:\mathbb{Z}^+\to \mathbb{R}^+$, $\pi_{s,m}:\mathbb{Z}^+\to \mathbb{R}^+$, $1\leq s \leq r$, $m \in \mathbb{Z}^+$ that satisfy (\ref{195}) and (\ref{196}) respectively and are such that for every fixed $m\in \mathbb{Z}$ it is verified

$$\mathlarger{\sum_{j=m}^{\infty}}D(j+1)h(j+1)\mathbb{D}^s_m(\Gamma_0)(j)<+\infty\text{ , for every }1\leq s \leq r.$$

\begin{remark}
Note that condition \textbf{(d7)} is a particular case of \textbf{(DIF,r)}, namely \textbf{(DIF,1)}, with  $\Gamma_1(j)=\gamma(j)$ and $\pi_{1,m}(j)= \prod_{p=m}^{j-1}\norm{A(p)}+\gamma(p)$. Similarly, the conditions on Lemma (\ref{193}) can be summarized as \textbf{(DIF,2)}.\\

Moreover, condition \textbf{(d3)} is equivalent to \textbf{(DIF,0)}, defining $j\mapsto \Gamma_0(j)$ as a function that uniformly (over $u$) dominates $j\mapsto f(j,u)$, or, in the notations of \textbf{(d3)}, $\Gamma_0(j)=\mu(j)$. This makes sense, because a $C^0$ class of differentiability for a diffeomorphism corresponds to an homeomorphism, which is verified by considering \textbf{(d3)} (in addition with \textbf{(d0)-(d5)}).
\end{remark}

\end{document}